\documentclass[11pt]{amsart}

\usepackage[
  textwidth=5.4in,
  textheight=8.2in,
  centering
]{geometry}

\usepackage{amsmath,amssymb,amsfonts}
\usepackage{amscd}
\usepackage{hyperref}

\newtheorem{theorem}{Theorem}
\newtheorem{lemma}{Lemma}
\newtheorem{remark}{Remark}
\newtheorem{corollary}{Corollary}
\newtheorem{definition}{Definition}

\newtheorem{question}{Question}

\DeclareMathOperator{\ind}{ind}

\title{On a Weak Form of the Topological Hedetniemi Conjecture}

\author{Hamid Reza Daneshpajouh}
\address{School of Mathematical Sciences, University of Nottingham Ningbo China, 199 Taikang East Road, Ningbo, 315100, China}
\email{Hamid-Reza.Daneshpajouh@nottingham.edu.cn}

\begin{document}

\maketitle

\begin{abstract}
The topological Hedetniemi conjecture asserts that the topological index of the product of two \(\mathbb{Z}/2\)-complexes is equal to the minimum of their respective indices.  Bui and Daneshpajouh studied a natural generalization of this conjecture to \(G\)-spaces and proved that this generalized conjecture fails whenever the group \(G\) is neither cyclic of prime-power order nor generalized quaternion. More precisely, for any such group \(G\), they constructed finite free \(G\)-simplicial complexes, each having topological index one, whose product has topological index zero.

The purpose of this note is to demonstrate that this discrepancy can be made arbitrarily large. Indeed, for every integer \(n \geq 1\), we show that there exist two finite free \(G\)-simplicial complexes, each of topological index $n$, whose product has topological index zero, provided that \(G\) is neither cyclic of prime-power order nor generalized quaternion. In particular, the corresponding weak form of the generalized topological Hedetniemi conjecture fails in the strongest possible sense for those groups.
\end{abstract}


\section{Introduction}

For finite graphs \(A\) and \(B\), their categorical product satisfies \(\chi(A\times B)\leq \min\{\chi(A),\chi(B)\}\), since each coordinate projection is a graph homomorphism. Hedetniemi's conjecture asserted that equality always holds. In terms of the Poljak--R\"{o}dl function 
\[f(n)= \min\{\chi(A\times B):\chi(A)=\chi(B)=n\},\] 
the conjecture is equivalent to the assertion that \(f(n)=n\) for every positive integer \(n\). El-Zahar and Sauer~\cite{el1985chromatic} proved Hedetniemi's conjecture for \(n\leq 4\). After remaining open for more than five decades, the conjecture was disproved by Shitov~\cite{shitov2019}, who showed that \(f(n)<n\) for all sufficiently large \(n\). Subsequent work~\cite{he2021hedetniemi, zhu2021relatively} pushed this strict inequality to smaller values of \(n\), and ultimately Tardif~\cite{tardif2023chromatic} showed the existence of a counterexample for the smallest possible such value, namely \(n=5\). A weaker asymptotic version, known as the \emph{weak Hedetniemi conjecture}, asserted that \(f(n)\to\infty\) as \(n\to\infty\); see~\cite{poljak1981arc}. This conjecture was recently disproved by Raanes~\cite{raanes2026counterexamples} by establishing the uniform bound \(f(n)\leq 4\) for every \(n\). Thus, both the original conjecture and its weak asymptotic variant are now known to be false.

Nevertheless, several intriguing problems inspired by Hedetniemi’s conjecture remain open. One of the most prominent is its topological analogue, known as the \emph{topological Hedetniemi conjecture}. Shortly before the original conjecture was disproved, Wrochna~\cite{wrochna2019} and Matsushita~\cite{matsushita2017z2} independently showed that Hedetniemi's conjecture would imply
\[
\operatorname{ind}_{\mathbb{Z}/2}(X \times Y)
=
\min\bigl\{
    \operatorname{ind}_{\mathbb{Z}/2}(X),
    \operatorname{ind}_{\mathbb{Z}/2}(Y)
\bigr\}
\]
for all finite free $\mathbb{Z}/2$-simplicial complexes $X$ and $Y$, where the product is equipped with the diagonal action. Wrochna~\cite{wrochna2019} conjectured that this equality holds. In~\cite{daneshpajouh2023hedetniemi}, an analogous statement was established for the homological index, a weaker invariant than the topological index. More recently, the natural extension of the topological Hedetniemi conjecture to finite free $G$-simplicial complexes was studied in~\cite{bui2024topological}. It was shown that this extension fails whenever $G$ is not a \emph{nice group}, where a nice group is either a cyclic group of prime-power order or a generalized quaternion group. More precisely, it was shown in~\cite[Theorem 1]{bui2024topological} that for every finite group $G$ outside these two classes, there exist finite free $G$-simplicial complexes $X$ and $Y$ such that
\[
\operatorname{ind}_G(X)
=
\operatorname{ind}_G(Y)
=
1,
\qquad\text{but}\qquad
\operatorname{ind}_G(X \times Y)
=
0.
\]
Here, for a free $G$-space $X$, the topological index is defined by
\[
\operatorname{ind}_G(X)
=
\min\bigl\{
    k \geq 0 :
    \text{there exists a $G$-equivariant map }
    X \longrightarrow E_kG
\bigr\},
\]
where $E_kG=G^{*(k+1)}$ is the $(k+1)$-fold topological join of $G$, equipped with the diagonal $G$-action. In particular, $E_k \mathbb{Z}/2$ is equivariantly homeomorphic to the antipodal $k$-sphere $S^k$.

Motivated by the Poljak--R\"{o}dl function, we introduce a numerical invariant that measures the behavior of the topological index under products.

\begin{definition}
\label{def:TG}
Let \(G\) be a nontrivial finite group. For \(n\geq 0\), define the \emph{topological Poljak--R\"{o}dl function} of \(G\) by
\[
T_G(n)=
\min\left\{
\operatorname{ind}_G(X\times Y):
\begin{array}{l}
X\text{ and }Y\text{ are finite free \(G\)-simplicial complexes},\\
\operatorname{ind}_G(X)=\operatorname{ind}_G(Y)=n
\end{array}
\right\},
\]
where \(X\times Y\) is equipped with the diagonal \(G\)-action.
\end{definition}

Since the coordinate projections \(X\times Y\to X\) and \(X\times Y\to Y\) are \(G\)-equivariant, monotonicity of the topological index gives \(T_G(n)\leq n\). In this terminology, the generalized topological Hedetniemi conjecture for \(G\) is precisely the assertion that \(T_G(n)=n\) for every \(n\geq 0\); see~\cite[Remark 2]{bui2024topological}. Moreover, the results of~\cite{bui2024topological} determine the first nontrivial value:
\[
T_G(1)=
\begin{cases}
1, & \text{if \(G\) is nice},\\
0, & \text{if \(G\) is not nice}.
\end{cases}
\]

The purpose of this note is to determine \(T_G(n)\) in every positive degree for groups that are not nice. Our main result shows that the failure detected in degree \(1\) persists in all higher degrees.

\begin{theorem}
\label{thm:main}
If $G$ is not a nice group, that is, if \(G\) is neither
cyclic of prime-power order nor generalized quaternion, then
\(
T_G(n)=0
\)
for every \(n\geq 1\).
\end{theorem}

An immediate consequence of Theorem~\ref{thm:main} is that the equivariant
index can behave strikingly nonadditively under joins. More precisely, for
every such group \(G\), the gap between the standard upper bound for the
index of a join and its actual index can be made arbitrarily large.

Recall that the join of two topological spaces \(X\) and \(Y\) is the
quotient
\[
X*Y=(X\times Y\times[0,1])/{\sim},
\]
where
\[
(x,y_1,0)\sim(x,y_2,0)
\qquad\text{and}\qquad
(x_1,y,1)\sim(x_2,y,1).
\]
We denote the equivalence class of \((x,y,t)\) by
\(
(1-t)x\oplus ty.
\)
If \(X\) and \(Y\) are \(G\)-spaces, then \(X*Y\) carries the diagonal
\(G\)-action
\[
g\cdot\bigl((1-t)x\oplus ty\bigr)
=
(1-t)(gx)\oplus t(gy).
\]

The canonical inclusions of \(X\) and \(Y\) into \(X*Y\) give the lower
bound
\[
\max\{\ind_G X,\ind_G Y\}\leq \ind_G(X*Y).
\]
On the other hand, if \(X\to E_nG\) and \(Y\to E_mG\) are \(G\)-maps,
then their join induces a \(G\)-map
\[
X*Y\longrightarrow E_nG*E_mG\cong E_{n+m+1}G.
\]
Consequently,
\[
\max\{\ind_G X,\ind_G Y\}
\leq \ind_G(X*Y)
\leq \ind_G X+\ind_G Y+1.
\]
Thus, it is natural to measure the failure of additivity by the
\emph{join defect}
\[
\delta_G(X,Y)
=
\ind_G X+\ind_G Y+1-\ind_G(X*Y).
\]

Csorba~\cite{csorba2007homotopy} constructed \(\mathbb{Z}/2\)-spaces
\(X\) and \(Y\) for which \(\delta_G(X,Y)=1\), and subsequently asked
whether the join defect can be arbitrarily large; see
\cite[Section~3.7]{csorba2005nontidy}. This gives an affirmative answer to the analogous join-defect question for every finite group that is neither cyclic of prime-power order nor generalized quaternion\footnote{Thus, the analogous join-defect question is resolved
affirmatively for every finite group covered by Theorem~\ref{thm:main},
while the cases \(G=\mathbb Z/p\) considered by
Csorba remain open.}. 

\begin{corollary}
\label{cor:large-join-defect}
Let \(G\) be a finite group that is neither cyclic of prime-power order
nor generalized quaternion. For every \(n\geq 1\), there exist finite free
\(G\)-simplicial complexes \(X\) and \(Y\) such that
\[
\ind_G X=\ind_G Y=n
\qquad\text{and}\qquad
\ind_G(X*Y)\leq n+1.
\]
In particular,
\[
\delta_G(X,Y)
=
\ind_G X+\ind_G Y+1-\ind_G(X*Y)
\geq n.
\]
Hence the join defect is unbounded.
\end{corollary}

\begin{proof}
By Theorem~\ref{thm:main}, there exist finite free \(G\)-simplicial complexes \(X\)
and \(Y\) satisfying
\[
\ind_G X=\ind_G Y=n
\qquad\text{and}\qquad
\ind_G(X\times Y)=0.
\]

Following the decomposition used in
\cite[Proof of Theorem~1.2]{daneshpajouh2023hedetniemi}, write points of
\(X*Y\) as \((1-t)x\oplus ty\) and consider the \(G\)-invariant open
subsets
\[
A=
\left\{
(1-t)x\oplus ty : t\neq \tfrac12
\right\}
\]
and
\[
B=
\left\{
(1-t)x\oplus ty : \tfrac13<t<\tfrac23
\right\}.
\]
These sets cover \(X*Y\).

The two components of \(A\) equivariantly deformation retract onto the
two ends of the join, so
\(
A\simeq_G X\sqcup Y
\). Therefore,
\[
\ind_G A
=
\ind_G(X\sqcup Y)
=
\max\{\ind_G X,\ind_G Y\}
=
n.
\]
Similarly, \(B\) equivariantly deformation retracts onto the midpoint
of the join, which is naturally \(G\)-homeomorphic to \(X\times Y\).
Hence
\[
\ind_G B=\ind_G(X\times Y)=0.
\]

Applying the subadditivity of the equivariant index to the invariant
open cover \(X*Y=A\cup B\), we obtain
\[
\ind_G(X*Y)
\leq \ind_G A+\ind_G B+1
=
n+0+1
=
n+1.
\]
It follows that
\[
\delta_G(X,Y)
=
\ind_G X+\ind_G Y+1-\ind_G(X*Y)
\geq n+n+1-(n+1)
=
n,
\]
as claimed.
\end{proof}

\section{Proof of the main result}
To prove Theorem~\ref{thm:main}, we first establish two auxiliary lemmas and introduce a necessary definition.

\begin{lemma}\label{lem:subgroups}
Let $G$ be a finite group which is not nice. Then there exist
prime-order subgroups $H,K\leq G$ such that
\[
gHg^{-1}\cap \ell K\ell^{-1}=\{e\}
\]
for all $g,\ell\in G$.
\end{lemma}

\begin{proof}
We separate two cases. Suppose first that $G$ is not a $p$-group. Choose two distinct primes $p$ and $q$ dividing $|G|$. By Cauchy's theorem, $G$ contains
subgroups
\[
H\cong C_p,\qquad K\cong C_q.
\]
Every conjugate of $H$ has order $p$, and every conjugate of $K$ has
order $q$. Thus
\[
gHg^{-1}\cap \ell K\ell^{-1}=\{e\}
\]
for all $g,\ell\in G$. Suppose now that $G$ is a $p$-group. Since $Z(G)$ is nontrivial,
Cauchy's theorem gives a subgroup
\(
H\leq Z(G)
\)
of order $p$. A finite $p$-group has a unique subgroup of order $p$~\cite[Remark 1]{bui2024topological}
if and only if it is cyclic or, when $p=2$, generalized quaternion.
Since $G$ is not nice, there is another subgroup
\[
K\leq G,\qquad |K|=p,\qquad K\neq H.
\]

Because $H$ is central, \(gHg^{-1}=H\) for every $g\in G$. Moreover, no conjugate of $K$ can equal $H$. Indeed, if
\(
\ell K\ell^{-1}=H,
\)
then, using the centrality of $H$,
\(
K=\ell^{-1}H\ell=H,
\)
a contradiction. Therefore $H$ and $\ell K\ell^{-1}$ are distinct
subgroups of order $p$, so
\[
H\cap \ell K\ell^{-1}=\{e\}.
\]
This proves the result.
\end{proof}

For a subgroup $H\leq G$ and a left $H$-space $A$, write
\(
G\times_H A
\)
for the induced $G$-space, where
\[
(gh,a)\sim(g,ha)
\qquad
(g\in G,\ h\in H,\ a\in A).
\]
The left $G$-action is
\(
u\cdot[g,a]=[ug,a].
\)

\begin{lemma}
\label{lem: index}
Let \(G\) be a finite group and \(H\leq G\). If \(X\) is a free \(H\)-space, then
\[
\operatorname{ind}_{G}(G\times_H X)=\operatorname{ind}_{H}(X).
\]
\end{lemma}

\begin{proof}
Let \(m=\operatorname{ind}_{H}(X)\). By definition, there exists an \(H\)-equivariant map \(f\colon X\to E_mH\). Inducing \(f\) from \(H\) to \(G\) gives a \(G\)-equivariant map
\[
\widetilde f\colon G\times_H X\longrightarrow G\times_H E_mH,
\qquad [g,x]\longmapsto [g,f(x)].
\]
The space \(G\times_H E_mH\) is a free \(m\)-dimensional \(G\)-complex. Since \(E_mG\) is \((m-1)\)-connected, the equivariant extension property~\cite[Lemma 6.2.2]{matousek2008using} yields a \(G\)-equivariant map \(G\times_H E_mH\to E_mG\). Composing these maps gives a \(G\)-equivariant map \(G\times_H X\to E_mG\). Therefore,
\[
\operatorname{ind}_{G}(G\times_H X)\leq \operatorname{ind}_{H}(X).
\]

Conversely, let \(n=\operatorname{ind}_{G}(G\times_H X)\). By definition, there exists a \(G\)-equivariant map \(\varphi\colon G\times_H X\to E_nG\). After restricting the action to \(H\), the space \(E_nG\) is a free \(n\)-dimensional \(H\)-space. Since \(E_nH\) is \((n-1)\)-connected, there exists an \(H\)-equivariant map \(\psi\colon E_nG\to E_nH\). The natural map \(i\colon X\to G\times_H X\), given by \(i(x)=[e,x]\), is \(H\)-equivariant. Hence the composition
\[
X\xrightarrow{i}G\times_H X\xrightarrow{\varphi}E_nG\xrightarrow{\psi}E_nH
\]
is \(H\)-equivariant. It follows that
\[
\operatorname{ind}_{H}(X)\leq \operatorname{ind}_{G}(G\times_H X).
\]
Combining the two inequalities proves the result.
\end{proof}

\begin{proof}[Proof of Theorem~\ref{thm:main}]
Choose prime-order subgroups \(H,K\leq G\) as in Lemma~\ref{lem:subgroups}, and define \(X_n=G\times_H E_nH\) and \(Y_n=G\times_K E_nK\). Since \(E_nH\) and \(E_nK\) are finite free simplicial complexes, the induced spaces \(X_n\) and \(Y_n\) inherit natural structures of finite free \(G\)-simplicial complexes. By Lemma~\ref{lem: index} and the standard identity \(\ind_L(E_nL)=n\)
for every nontrivial finite group \(L\)
\cite[Proposition~6.2.4]{matousek2008using}, we obtain
\[
\ind_G(X_n)=\ind_H(E_nH)=n
\qquad\text{and}\qquad
\ind_G(Y_n)=\ind_K(E_nK)=n.
\]

It remains to prove that the product has index zero. Define \(G\)-equivariant maps \(q_H\colon X_n\to G/H\) and \(q_K\colon Y_n\to G/K\) by \(q_H([g,x])=gH\) and \(q_K([g,y])=gK\), respectively. Their product gives a \(G\)-equivariant map
\[
q_H\times q_K\colon X_n\times Y_n\longrightarrow G/H\times G/K,
\]
where \(G\) acts diagonally on \(G/H\times G/K\).

We claim that this diagonal action is free. Indeed, if \(a\in G\) fixes \((gH,\ell K)\), then \(agH=gH\) and \(a\ell K=\ell K\). Equivalently,
\[
a\in gHg^{-1}\cap \ell K\ell^{-1}.
\]
By the choice of \(H\) and \(K\), the intersection \(gHg^{-1}\cap \ell K\ell^{-1}\) is trivial for every \(g,\ell\in G\). Hence \(a=e\), proving that \(G/H\times G/K\) is a free finite \(G\)-set.

Every free \(G\)-set admits a \(G\)-equivariant map to \(G\), where \(G\) acts on itself by left multiplication. For completeness, choose one element \(s_i\) from each \(G\)-orbit in \(G/H\times G/K\), and define \(\varphi\colon G/H\times G/K\to G\) by \(\varphi(as_i)=a\). This is well defined because the action is free: if \(as_i=bs_i\), then \(b^{-1}a\) fixes \(s_i\), and therefore \(a=b\). Moreover, \(\varphi\) is \(G\)-equivariant by construction. Consequently, the composition
\[
X_n\times Y_n\xrightarrow{\,q_H\times q_K\,}G/H\times G/K\xrightarrow{\,\varphi\,}G=E_0G
\]
is a continuous \(G\)-equivariant map. Thus
\(
\ind_G(X_n\times Y_n)=0
\).
\end{proof}

\begin{remark}
The proof by Bui and Daneshpajouh~\cite{bui2024topological} that
\(T_G(1)=0\) for \(G\) not nice is combinatorial. More precisely, they
construct two finite free \(G\)-posets, each having cross-index \(1\),
whose product has cross-index \(0\). Recall that the cross-index is a
combinatorial analogue of the topological index; see~\cite{simonyi2013colourful, bui2026mapping, daneshpajouh2025box} for further background and its applications. It is worth mentioning that this construction can also be adapted to the cross-index setting to show that, for every \(n\geq 1\), there exist two finite free \(G\)-posets of cross-index \(n\) whose product has cross-index zero. Thus, the same strong failure occurs at the combinatorial level.
\end{remark}

\section{Open Problems}

We conclude with two natural questions concerning the groups not covered by
Theorem~\ref{thm:main}. Recall that a finite group \(G\) is called
\emph{nice} if it is either cyclic of prime-power order or generalized
quaternion.

\begin{question}
\label{q:weak-topological-hedetniemi}
Let \(G\) be a nice finite group. Is it true that
\(
\lim_{n\to\infty} T_G(n)=\infty
\)?
\end{question}

\begin{question}
\label{q:unbounded-join-defect}
Let \(G\) be a nice finite group. Is the join defect
\[
\delta_G(X,Y)
=
\ind_G X+\ind_G Y+1-\ind_G(X*Y)
\]
unbounded as \(X\) and \(Y\) range over finite free \(G\)-simplicial
complexes? Equivalently, for every \(n\geq 1\), do there exist finite free
\(G\)-simplicial complexes \(X\) and \(Y\) such that
\[
\ind_G X+\ind_G Y+1-\ind_G(X*Y)\geq n?
\]
\end{question}
As shown in the proof of Corollary~\ref{cor:large-join-defect}, if the
sequence \(\{T_G(n)\}_{n\geq 1}\) is bounded for some $G$, then the join defect is
unbounded. Consequently, Question~\ref{q:unbounded-join-defect} has an
affirmative answer for that \(G\).

\bibliographystyle{amsplain}
\bibliography{biblio}

\end{document}